\documentclass[reqno,12pt]{amsart}
\usepackage{color}
\usepackage{mathtools}
\usepackage{amsfonts}
\usepackage{amssymb}
\usepackage{amsthm}
\usepackage{newlfont}
\usepackage{graphicx}
\usepackage{float}
\usepackage{amsmath}
\usepackage{geometry}
\usepackage{tikz}
\usetikzlibrary{matrix,shapes,arrows,positioning,chains}

\numberwithin{equation}{section}
\newtheorem{thm}{Theorem}[section]

\newtheorem{prop}{Proposition}[section]

\newtheorem{example}{Example}[section]

\newcommand{\ed}{\end {document}}

\newcommand{\ka}{\kappa}

\begin{document}

\title{On a parabolic sine-Gordon model}

\author[X.Y. Cheng]{Xinyu Cheng}
\address{X.Y. Cheng, Department of Mathematics, University of British Columbia, Vancouver, BC V6T 1Z2, Canada}
\email{xycheng@math.ubc.ca}

\author[D. Li]{ Dong Li}
\address{D. Li, Department of Mathematics, the Hong Kong University of Science \& Technology, Clear Water Bay, Kowloon, Hong Kong}
\email{mpdongli@gmail.com}

\author[C.Y. Quan]{Chaoyu Quan}	
\address{C.Y. Quan, SUSTech International Center for Mathematics, Southern University of Science and Technology,
	Shenzhen, P.R. China}
\email{quancy@sustech.edu.cn}

\author[W. Yang]{Wen Yang}
\address{\noindent W. Yang,~Wuhan Institute of Physics and Mathematics, Chinese Academy of Sciences, P.O. Box 71010, Wuhan 430071, P. R. China; Innovation Academy for Precision Measurement Science and Technology, Chinese Academy of Sciences, Wuhan 430071, P. R. China.}
\email{wyang@wipm.ac.cn}

\begin{abstract}
We consider a parabolic sine-Gordon model with periodic boundary conditions. 
We prove a fundamental maximum principle which gives a priori uniform control of the solution.
In the one-dimensional
case we classify all bounded steady states and exhibit some explicit solutions. 
For the numerical discretization we employ first order  IMEX,  and second order BDF2 discretization without any additional stabilization term. We rigorously prove the energy stability of the numerical schemes under nearly sharp and quite mild time step constraints. We demonstrate the striking
similarity of the parabolic sine-Gordon model with the standard Allen-Cahn equations with double
well potentials. 
\end{abstract}

\maketitle

\section{Introduction}
In this work we are concerned with  the following parabolic sine-Gordon equation:
\begin{equation}
\label{1.sg}
\begin{cases}
\partial_tu=\ka^2\Delta u+\sin u,\quad (t,x)\in (0,\infty)\times \Omega,\\
u\bigr|_{t=0}=u_0,
\end{cases}
\end{equation}
where $\ka^2$ is the diffusion constant and $\Omega$ is either a  periodic torus $\mathbb T=[-\pi,\pi]$ in 1D or the torus $\mathbb T^2=[-\pi,\pi]\times[-\pi,\pi]$ in 2D. The unknown function
$u:\; \Omega \to \mathbb R$ typically represents the concentration difference in the phase field
context.  For smooth solutions, the basic energy associated with \eqref{1.sg} is 
\begin{equation}
\label{1.sg-energy}
E(u)=\int_\Omega\left(\frac{\ka^2}{2}|\nabla u|^2+\cos u\right)dx.
\end{equation}
The fundamental energy conservation law takes the form
\begin{equation}
\label{1.energy-id}
\frac{d}{dt}E(u)+\int_\Omega|\partial_tu|^2dx=0.
\end{equation}
It follows that
\begin{equation}
\label{1.e-de}
E(u(t))\le E(u(s)),\quad \forall\, t\ge s,
\end{equation}
which gives a priori control of the homogeneous $\dot H^1$-norm of the solution. Better estimates are also available. For example assuming $u_0$ is bounded, then by using the fact
that the nonlinear term $\sin u$ is bounded by 1, one can show that the solution remains bounded for all finite time. Bootstrapping from this then easily yields global wellposedness and regularity of the solution. Somewhat akin to the equation \eqref{1.sg} is the following slightly more general
model
\begin{align} \label{vEq}
\partial_{\tau} v = \kappa^2 \Delta v + \gamma \sin \beta v,
\end{align}
where $\beta>0$, $\gamma>0$ are parameters, and we denote by $\tau$ the time variable. One can rewrite \eqref{vEq} as
\begin{align}
\frac {\partial ( \beta v) } {\partial (\gamma \beta \tau ) }
= \frac {\kappa^2} {\gamma \beta} \Delta (\beta v) 
+ \sin (\beta v).
\end{align}
Consequently a change of variable $u=\beta v$, $t= \gamma \beta \tau$ transforms \eqref{vEq}
into the standard form \eqref{1.sg}.

The classical one dimensional sine-Gordon equation
\begin{align} \label{wSG0}
\partial_{tt}\phi - \partial_{xx} \phi = - \sin \phi
\end{align}
dates back at least to Frenkel and Kontorova \cite{fk1939} who considered the motion of
a slip in an infinite chain of atoms lying on top of a given fixed chain of
alike atoms. To study the propagation of the slip they obtained a difference differential equations
which was approximated by the sine-Gordon equation \eqref{wSG0}. 
In the realm of nonlinear
field theory, the sine-Gordon equation 
\begin{align} \label{wSG1}
\partial_{tt}\phi - \partial_{xx} \phi = -m^2 \sin \phi
\end{align}
arises as one of the simplest intrinsically nonlinear theories. The classical
point-like particle theories suffer divergence problems such as the well-known
self-energy problem of electrodynamics. It was realized that (cf. the discussion on pp. 260 of \cite{b1971}) one must consider nonlinear field theory in order to predict both the existence
and dynamics of extended elementary particles. Instead of augmenting a linear theory with
the addition of a nonlinear term ad hoc, a more natural way is to postulate a field whose
target is a nonlinear manifold. In this context the simplest topologically nontrivial manifold
is the standard $1$-sphere, i.e. the set of real numbers modulo $2\pi$. The sine-Gordon
Lagrangian density is postulated as
\begin{align}
\mathcal L_{\mathrm SG} = \frac 12 \bigl(\phi_t^2 -\phi_x^2 - 2m^2(1-\cos \phi) \bigr).
\end{align}
One should note that for small $\phi$, we have $2m^2(1-\cos \phi) \approx m^2 \phi^2$, i.e.
we can recover the usual Klein-Gordon Lagrangian density
\begin{align}
\mathcal L_{\mathrm{KG}} = \frac 12 \bigl( \phi_t^2 -\phi_x^2 -m^2 \phi^2).
\end{align}
The significance of the term $2m^2 (1-\cos \phi)$ is that it is the simplest periodic
function of $\phi$ which coincides with the Klein-Gordon case in the low-amplitude limit.
The periodicity of the nonlinear term has the effect of restricting the range of $\phi$ to be
the $1$-sphere.  Already within the limits of one-dimensional classical field theory, the sine-Gordon
equation \eqref{wSG1} gave a very good picture of the interaction of elementary particles and
the existence of bound states, and in particular it may lead to to solutions with the collisional properties of solitons \cite{ps1962,t1989}. The sine-Gordon equation has also been studied as a model in the theory of crystal dislocations, the motion of rigid pendular attached to stretched wire, and splay waves in lipid membranes and magnetic flux on Josephson line. We refer the interested readers to \cite{b1971,bgn1982,ceg1975,dh2000,f2012,fk1939,h1977,n1983,r1970} and the references therein
for more extensive discussions.

\begin{figure}[!h] \label{figure001}
\centering
\includegraphics[width=0.98\textwidth]{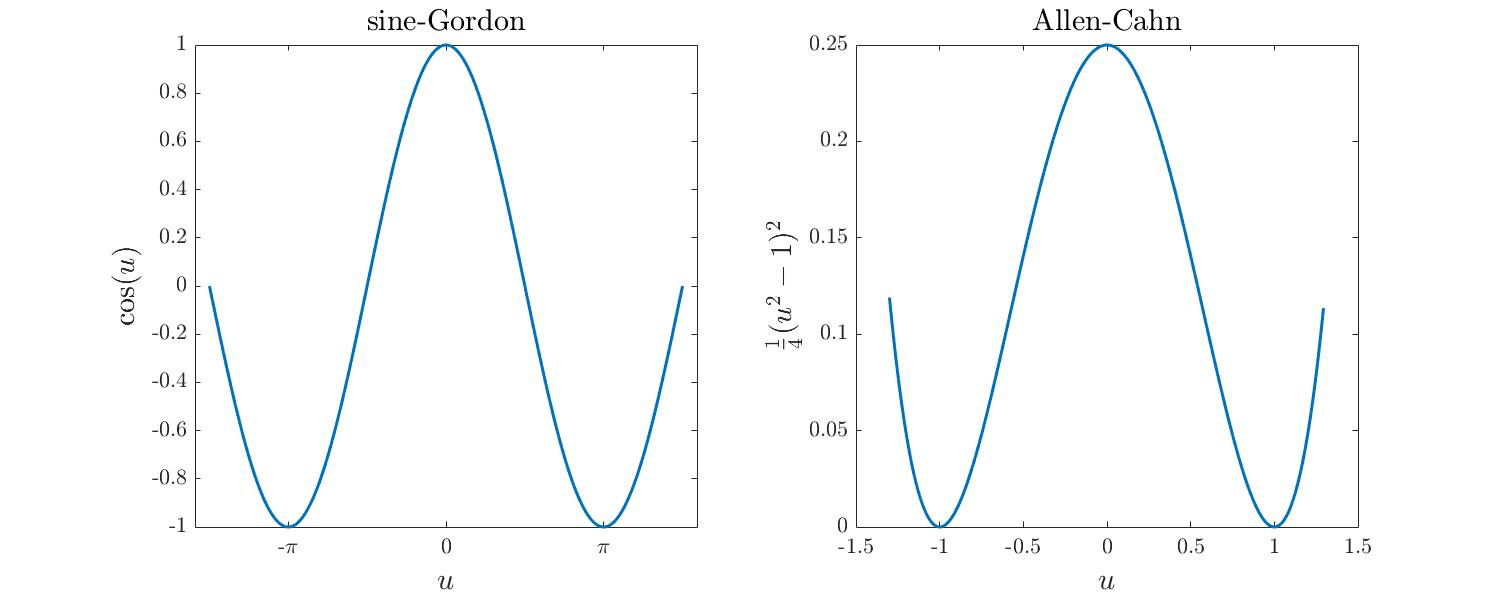}
\caption{\small Comparison of the double-well potentials $F(u)$ between the sine-Gordon model and the Allen-Cahn model with polynomial potential. }
\end{figure}

Our parabolic sine-Gordon model \eqref{1.sg} can be viewed as the parabolic version of the usual
wave-type sine-Gordon in the phase-field context.  It naturally arises from the gradient flow
of the energy functional 
\begin{align} \label{Esg01}
\mathcal E_{\mathrm{SG}} (u) = 
\int_\Omega\left(\frac{\ka^2}{2}|\nabla u|^2+\cos u\right)dx
\end{align}
in $L^2( \Omega)$.  Note that if we consider the $H^{-1}$-gradient flow of 
\eqref{Esg01}, then we obtain the model:
\begin{align} \label{Esg02}
\partial_t u = \Delta(  -\kappa^2 \Delta u - \sin u)
\end{align}
which is akin to the usual Cahn-Hilliard equation. The model \eqref{Esg02} will be studied
elsewhere. As it turns out, the potential term $F_{\mathrm{SG}}(u)=\cos u$ looks qualitatively similar to the usual
double well potential $F_{\mathrm{st}}(u)=(u^2-1)^2/4$ for $|u|=O(1)$, cf. Figure \ref{figure001}.  For this reason it is natural to speculate that there is some natural one-to-one
correspondence between solutions to the parabolic sine-Gordon equation \eqref{1.sg} and
the usual Allen-Cahn equations.  On the other hand, the parabolic
sine-Gordon equation is quite appealing for both analysis and simulation since its nonlinearity has bounded derivatives of all orders. As a matter of fact,  the potential $F_{\mathrm{SG}}(u)
=\cos u$ is one of the handy choices
for testing and benchmarking algorithms in the computational phase field community.
The purpose of this work is to initiate the study of \eqref{1.sg} and establish a number of
basic facts for solutions to \eqref{1.sg}. More importantly we prove the fundamental monotonicity laws
for the solutions, characterize and classify one-dimensional steady states, and analyze the stability of several prototypical numerical discretization schemes implemented on this model. Remarkably due to the very benign
nonlinear structure one can prove optimal energy stability results without resorting to any $L^{\infty}$-maximum
principle. This feature is quite appealing and we expect future development of our analysis on
the model \eqref{Esg02} under slightly more stringent time step constraints.

The rest of this paper is organized as follows. In Section \ref{append:mp_SG} we prove
a fundamental maximum principle of the parabolic sine-Gordon model in all dimensions. 
In Section 3 we classify all bounded steady state solutions in one dimension. In Section 4
we analyze two numerical discretization schemes and prove optimal energy stability. 
In Section 5 we carry out several numerical experiments showcasing the striking similarity
of the parabolic sine-Gordon model and the  usual Allen-Cahn equation.  In the last section
we give concluding remarks.

\section{Maximum principle}\label{append:mp_SG}
In this section we prove a useful maximum principle for \eqref{1.sg} on the the torus
$\mathbb T^d=[-\pi,\pi]^d$ for all dimensions $d\ge 1$.

\begin{thm}[Globalwell posedness and maximum principle]
Let $\kappa>0$ and consider \eqref{1.sg} on $\mathbb T^d=[-\pi, \pi]^d$, $d\ge 1$.
Suppose $\|u_0\|_{\infty} \le \pi$. Then there exists a unique global solution $u$
to \eqref{1.sg} which is smooth for all $t>0$.  Furthermore we have
\begin{align} \label{tP2.1}
\sup_{0\le t <\infty} \| u(t,\cdot )\|_{\infty} \le \pi.
\end{align}
\end{thm}
\begin{proof}
We begin by noting that since the nonlinear term $\sin u$ has uniformly bounded derivatives of all orders, it is utterly standard to obtain the global wellposedness and regularity of the solution to \eqref{1.sg} and thus we focus on the proof of \eqref{tP2.1}.

We first prove \eqref{tP2.1} under the assumption that $\|u_0\|_{\infty}<\pi$. Since $\|u_0\|_{\infty}<\pi$, by using smoothing estimate we may assume with no loss that $u_0$ is smooth and still satisfy $\|u_0\|_{\infty} <\pi$.  Fix $0<\varepsilon\ll 1$ which will be taken to tend to zero later and consider $v(t,x) = u(t,x) - \pi -\varepsilon$. We claim the following:
	\begin{align} \label{vtx01}
		\sup_{t\ge 0} \max_{x\in \mathbb T^d} v(t,x) \le 0.
	\end{align}
	Assume the claim is not true, then we can find  $t_*>0$  such that
	\begin{align}
		&\max_{x\in \mathbb T^d} v(t_*, x) =0,  \\
		&\max_{x\in \mathbb T^d} v(t, x) >0, \qquad t\in (t_*, t_*+\delta_0), \label{vtx03}
	\end{align}
	where $\delta_0>0$ is a sufficiently small constant. Assume $v(t_*,x)$ takes its maximum
	at some $x_*\in \mathbb T^d$. Note that $u(t_*,x_*)=\pi +\varepsilon$. We have
	\begin{align}
		\partial_t v(t, x_*) \Bigr|_{t=t_*} \le \sin (\pi+\varepsilon) =- \sin \varepsilon<0.
	\end{align}
	By continuity, we can find $\delta_*>0$ sufficiently small such that
	if $|t-t_*|+|y-x_*| <3\delta_*$, then
	\begin{align}
		\partial_t v (t,y) <-\frac 12 \sin \varepsilon<0.
	\end{align}
	In particular, for $|y-x_*|<\delta_*$ and $|t-t_*|<\delta_*$, we have
	\begin{align} \label{vtx05a}
		v(t,y) \le 0.
	\end{align}
	Now for $t=t_*$ and any $x\in \mathbb T^d$ such that $v(t,x)<0$, we can find a neighborhood
	$N_x$ and $\delta_x>0$ such that for any $t\in [t_*, t_*+\delta_x]$, $y \in N_x$,
	\begin{align} \label{vtx05b}
		v(t,y)<0.
	\end{align}
	By a covering argument using \eqref{vtx05a} and \eqref{vtx05b} (if $v(t_*,x)=0$ we use
	\eqref{vtx05a}, and if $v(t_*,x)<0$ we use \eqref{vtx05b}), we obtain for $t_*\le t \le t_*+\delta_1$ and $\delta_1>0$
	sufficiently small,
	\begin{align}
		\max_{t_*\le t\le t_*+\delta_1} \max_{x\in \mathbb T^d} v(t,x) \le 0.
	\end{align}
	This clearly contradicts \eqref{vtx03} and thus the claim \eqref{vtx01} holds.
	
	By taking $\varepsilon\to 0$, we obtain $u(t,x) \le \pi$ for all $x\in \mathbb T^d$ and $t\ge 0$.
	By working with $-u$ we obtain $u(t,x) \ge -\pi$ for all $x\in \mathbb T^d$ and $t\ge 0$.
	The desired result then follows easily.
	
Finally we show how to get \eqref{tP2.1} under the assumption that $\|u_0\|_{\infty} \le \pi$.
It suffices for us to show for any finite $T>0$, 
\begin{align} \label{Ft001}
\sup_{0\le t \le T} \|u(t,\cdot )\|_{\infty} \le \pi.
\end{align}
The trick is to use stability.  For $n\ge 1$, consider \eqref{1.sg} with initial data $u_0^{(n)} = (1-2^{-n-1}) u_0$ and denote the corresponding solution as $u^{(n)}$. Apparently we have
\begin{align}
\sup_{0\le t \le T} \| u^{(n)}(t,\cdot) \|_{\infty} \le \pi, \quad\forall\, n\ge 1.
\end{align}
Observe that
\begin{align}
u(t) -u^{(n)}(t) =e^{\kappa^2 t \Delta}
( u_0 - u_0^{(n)} ) + \int_0^t e^{\kappa^2 (t-s) \Delta}
( \sin u(s) -\sin u^{(n)}(s) ) ds.
\end{align}
From this one can extract the $L^{\infty}$-stability estimate of $u-u^{(n)}$. In particular, it is not difficult to check that  
\begin{align}
\sup_{0\le t \le T} \| u(t,\cdot ) -u^{(n)}(t,\cdot) \|_{\infty} \to 0, 
\end{align}
as $n\to \infty$.  Thus \eqref{Ft001} follows.	
\end{proof}

\section{Classification of the steady states in 1D}

In this section we consider bounded steady states of the sine-Gordon equation,
\begin{equation}
\label{a.sg}
\ka^2u''+\sin u=0,\quad x\in\mathbb{R}.
\end{equation}
Note that here we consider the whole real axis for generality. Functions on the torus $\mathbb T
=[-\pi, \pi]$ can be naturally identified as a periodic function on $\mathbb R$. 

\begin{prop}[Rigidity of the solution to \eqref{a.sg}]
\label{Rigid1}
The following hold.
\begin{itemize}
\item \underline{Even reflection}. Suppose $\kappa>0$, and for some $r_0>0$ we have
\begin{align}
\kappa^2 u^{\prime\prime} +\sin u =0, \qquad \forall\, -r_0<x<0,
\end{align}
where $u \in C^2( (-r_0, 0) )$ and we assume $\lim\limits_{x\to 0-} u^{\prime}(x)=0$.
Define $u(x) =u(-x)$ for $0<x<r_0$. Then it holds that $u \in C^{\infty}((-r_0,r_0) )$ with $u^{\prime}(0)=0$ and solving the same equation on the
whole interval.

\item \underline{Odd reflection}. Suppose $\kappa>0$, and for some $r_0>0$ we have
\begin{align}
\kappa^2 u^{\prime\prime} +\sin u =0, \qquad \forall\,  0<x<r_0.
\end{align}
where $u \in C^2( (0, r_0) )$ and we assume $\lim\limits_{x\to 0+} u(x)=0$.
Define $u(x) =-u(-x)$ for $-r_0<x<0$. Then it holds that $u \in C^{\infty}((-r_0,r_0) )$ with $u(0)=0$ and solving the same equation on the
whole interval.
\end{itemize}
\end{prop}

\begin{proof}
We shall only prove the first case as the second case is similar. First it is not difficult to that $u$ has
bounded derivatives in $[-r_0/2, 0)$ which can be extended to $0$ from the left. The
extended $u$ satisfies the equation on $(-r_0, 0) \cap (0,r_0)$. Furthermore the
equation also holds at $x=0$ up to third order derivatives. Then we can bootstrap the regularity
of $u$ by using the equation and conclude that $u \in C^{\infty}$.
\end{proof}

It is easy to see that if $u$ is a solution to \eqref{a.sg}, then for any integer $m\in\mathbb Z$ and
$x_0\in \mathbb R$, $u(\cdot+x_0)+2m \pi$ is still a solution to \eqref{a.sg}. 
Therefore with no loss we can consider solutions $u$ with $|u(0)|\le\pi.$

 Multiplying \eqref{a.sg} by $u'$, we derive 
\begin{equation}
\label{a.id}
\frac12\ka^2(u')^2=C+\cos u,
\end{equation}
where $C\ge -1$ is a constant. 
Concerning the solution of \eqref{a.sg}, we have the following result.
\begin{prop}
\label{pra.1}
Let $u$ be a bounded solution to \eqref{a.sg} with $|u(0)|\le\pi$ and $C\ge -1$ be the constant defined in \eqref{a.id}, then the following hold.
\begin{enumerate}
\item [(1)] For $C>1$, there does not exist any bounded solution.
\item [(2)] If $C=-1$, then $u\equiv0.$
\item [(3)] If $C=1$, then $u=\pm 2\arcsin\tanh\left(\frac{x}{\ka}+c\right)
$ for some constant $c\in\mathbb R$ or $u\equiv \pm \pi$.
\item [(4)] If $-1<C<1$, then $u$ is a periodic function and $\|u\|_{\infty}<\pi$.

\end{enumerate}	
\end{prop}

\begin{proof}
We proceed in several steps.

(1) If $C>1$, then $u'$ never changes its sign and it implies that $u$ is either an increasing or a decreasing function. In addition $|u'|$ has a positive lower bound, it implies that $u$ is unbounded. Thus, there is no bounded solution for $C>1$.

(2) If $C=-1$, then 
\begin{equation}
\cos u-1=\frac12\ka^2(u')^2\ge 0,\quad \forall\, x\in \mathbb R.
\end{equation}
Since $|u(0)|\le \pi$, it follows that $u\equiv 0$.

(3) In the case $C=1$, it is easy to check that $u\equiv\pi$ or $-\pi$ is always a solution.  On the other
hand when  $ |u(0)| <\pi $, we can explicitly solve \eqref{a.id} and get
$$u=\pm 2\arcsin\tanh\left(\frac{x}{\ka}+c\right),$$
where $c$ is a constant.

(4) In the case $C\in(-1,1)$, note that $\arccos(-C) \in (0, \pi)$.  By \eqref{a.id} and the assumption
that $|u(0)|\le \pi$,  we have
$|u(0)| \le \arccos(-C)$.  We first discuss the case $|u(0)| <\arccos(-C)$.  In this case \eqref{a.id}
simplifies to
\begin{align} \label{eT2.0}
u^{\prime} = \pm \frac {\sqrt 2} {\kappa} \sqrt{C+\cos u}.
\end{align}
With no loss we consider the case $u^{\prime}(0)>0$ and work with the ODE:
\begin{align} \label{eT2.1}
u^{\prime} =  \frac {\sqrt 2} {\kappa} \sqrt{C+\cos u}.
\end{align}
It is not difficult to solve \eqref{eT2.1} on a maximal interval $[x_-, x_+]$ such that
$-\infty<x_-<0<x_+<\infty$, $u^{\prime}(x_-)=u^{\prime}(x_+)=0$.  Furthermore
$u^{\prime}(x)>0$ for any $x_-<x<x_+$. Clearly on the interval $(x_-,x_+)$, $u$ is
a smooth solution to $\kappa^2 u^{\prime\prime}+\sin u=0$. By 
Proposition \ref{Rigid1}, we can uniquely extend this solution (via repeated reflections)
to the whole real axis. The obtained solution is clearly periodic and satisfies $\|u\|_{\infty}\le
\arccos (-C)$. 

The case $u^{\prime}(0)<0$ is similar since we can work with $-u$ and repeat the argument. 

Finally we consider the case $|u(0)|=\arccos(-C)$.  With no loss we consider $u(0)=
\arccos(-C)$.  Clearly $u^{\prime}(0)=0$. By using $u^{\prime\prime}=-\frac 1 {\kappa^2}
\sin u$, we get $u^{\prime\prime}(0)<0$. We can find $x_0<0$ sufficiently close to $0$ such that 
$u(x_0)<\arccos(-C)$. Starting from the initial value $u(x_0)$, we can then work with the ODE 
\begin{align} 
u^{\prime} =  \frac {\sqrt 2} {\kappa} \sqrt{C+\cos u}, 
\end{align}
and solve it on a maximal interval $[r_-, 0]$, where $-\infty<r_-<0$ and $u^{\prime}(r_-)=
u^{\prime}(0)=0$.
We then apply Proposition  \ref{Rigid1} to obtain the  periodic solution
on the real axis.
\end{proof}

\begin{figure}[!h]
\centering
\includegraphics[width=0.98\textwidth,clip,trim={0 1in 0 0.2in}]{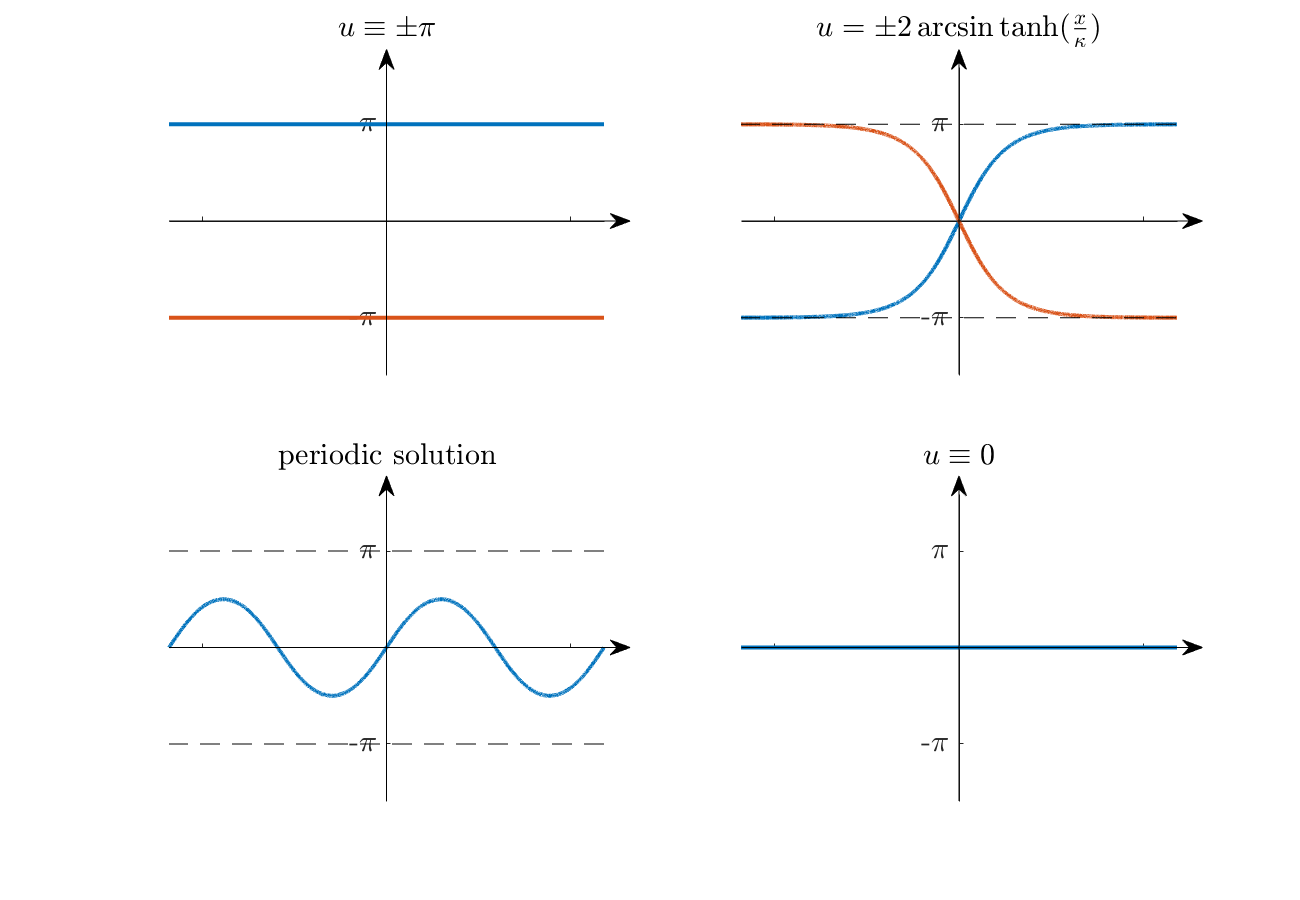}
\caption{\small Different bounded steady states of 1D sine-Gordon equation. Here the periodic
solution corresponds to  $\ka = 0.5$ and $C=0$ in \eqref{a.id}. }
\end{figure}

\section{Numerical schemes}
\subsection{First-order IMEX scheme}
We consider the following first order implicit-explicit (IMEX) scheme:
\begin{equation}\label{eq:sch1}
\frac{u^{n+1}-u^n}{\tau}=\ka^2\Delta u^{n+1}+\sin (u^n),
\end{equation}
where $\tau>0$ denotes the time step, and $u^n:\; \Omega =\mathbb T^d \to \mathbb R$ corresponds to the numerical solution
computed at step $n$. One can rewrite \eqref{eq:sch1} as
\begin{align}
(1-\kappa^2 \Delta) u^{n+1} = u^n +\tau \sin (u^n)
\end{align}
or
\begin{align}\label{eq:4.3.0}
u^{n+1}= (1-\kappa^2 \Delta)^{-1} (u^n +\tau \sin (u^n ) ).
\end{align}
The solvability is not an issue once the initial data $u^0$ is given. In more practical numerical
computations one can employ spectral methods to compute the operator $(1-\kappa^2 \Delta)^{-1}$
very efficiently. 

\begin{thm}[Discrete maximum principle]\label{thm4.1.0}
Consider the scheme \eqref{eq:sch1}. Assume $\|u^0\|_\infty\le \pi$. 
If $0< \tau\le 1$, then $\|u^n\|_\infty\le \pi$ for all $n\ge 1$.
\end{thm}
\begin{proof}
Consider the function $f_{\tau}(z) = z+\tau\sin(z)$. It is easy to see that
if $0<\tau\le 1$, $f_{\tau}'(z)\ge 0$ for all $z\in \mathbb R$.
Thus, $\max_{|z|\le \pi} |f_\tau(z)| \le \pi$.
The desired result then follows from \eqref{eq:4.3.0} together with a maximum principle for the operator $(1-\kappa^2 \Delta)^{-1}$, i.e., $\|(1-\kappa^2 \Delta)^{-1}\|_{L^\infty\rightarrow L^\infty}\le 1$, cf. \cite{Li2013a,Li2013b}.
\end{proof}

Concerning the energy dissipation, we have the following result. 
Note that the time step constraint is $0<\tau\le 2$, which is wider than the one given by Theorem \ref{thm4.1.0}. 
This is because we do not need to use the maximum principle in the proof. 
\begin{thm}[Energy dissipation]\label{thm3.1}
If $0<\tau\leq 2$, then the following energy dissipation law holds for the scheme \eqref{eq:sch1}:
\begin{equation}
E(u^{n+1}) \le  E(u^{n}), \quad \forall \, n\geq 0,
\end{equation}
where 
\begin{equation}
E(u^n) = \frac{\ka^2}2\|\nabla u^{n}\|^2 + \int_{\Omega} \cos(u^{n}) dx.
\end{equation}
Here $\|\cdot\|=\|\cdot\|_2$. 
\end{thm}
\begin{proof}
Multiplying \eqref{eq:sch1} by $u^{n+1}-u^n$ and integrating over $\Omega$, we have
\begin{equation}
\frac{1}{\tau} \| u^{n+1}-u^n \|^2
=-\ka^2\langle \nabla u^{n+1},\nabla u^{n+1}-\nabla u^n\rangle_\Omega+\langle\sin(u^n)(u^{n+1}-u^n),1\rangle_\Omega,
\end{equation}
where $\langle \cdot,\cdot\rangle_\Omega$ denotes the $L_2$ inner product on $\Omega$, i.e.
\begin{align}
\langle f , g \rangle_{\Omega} = \int_{\Omega} f(x) g(x) dx,
\qquad\text{for $f$, $g:\Omega \to \mathbb R$}.
\end{align}
It follows that
\begin{align}\label{ineq:3.4}
&\frac{\ka^2}{2}\|\nabla u^{n+1}\|^2 + \int_{\Omega} \cos(u^{n+1}) dx- \frac{\ka^2}{2}\|\nabla u^{n}\|^2 - \int_{\Omega} \cos(u^{n})dx  \notag \\
\le &\; -  \left\langle \frac 1 \tau -\frac 1 2\cos(\xi^n), (u^{n+1}-u^n)^2\right\rangle_\Omega,
\end{align}
where $\xi^n  $ is some function between $u^n$ and $u^{n+1}$.
Obviously, when $0<\tau\leq 2$, the right-hand sider of the above inequality \eqref{ineq:3.4} is always non-positive.
\end{proof}

\subsection{Second-order BDF2 scheme}
We consider the following BDF2 scheme of the sine-Gordon equation
\begin{equation}\label{eq:sch2}
\frac{3u^{n+1}-4u^n+u^{n-1}}{2\tau}=\ka^2\Delta u^{n+1}+2 \sin (u^n)-\sin(u^{n-1}), \qquad n\ge 1.
\end{equation}
To kick start the scheme one can compute $u^1$ using a first order scheme such as \eqref{eq:sch1}. 
We have the following modified energy dissipation law for this second-order scheme.
\begin{thm}[Energy dissipation]\label{thm3.2}
If $0<\tau\leq \frac12$, then the energy dissipation law holds for scheme \eqref{eq:sch2}
\begin{equation}
\widetilde E(u^{n+1}) \leq  \widetilde E(u^{n}), \quad \forall\, n\ge 1,
\end{equation}
where
\begin{align}
\widetilde E(u^{n})  &= E(u^n)+\frac 1{4\tau} \| u^{n}-u^{n-1}\|^2 \notag \\
&=\frac 12 \kappa^2 \|\nabla u^{n}\|^2 +\int_{\Omega} \cos(u^{n}) dx+\frac 1{4\tau} \| u^{n}-u^{n-1}\|^2
\end{align}
is the modified energy.
\end{thm}

\begin{proof}
We first observe that
\begin{align} \label{4.12t0}
\frac{3u^{n+1}-4u^n+u^{n-1}}{2\tau}
= \frac {u^{n+1}-u^n}{\tau} + \frac{u^{n+1}-2u^n+u^{n-1}}{2\tau}.
\end{align}
Multiplying \eqref{eq:sch2} by $u^{n+1}-u^n$ and integrating over $\Omega$, we obtain
\begin{equation}\label{eq:bdf2int}
\left\langle\frac{3u^{n+1}-4u^n+u^{n-1}}{2\tau},u^{n+1}-u^n \right\rangle_\Omega =\left\langle\ka^2\Delta u^{n+1}+2 \sin (u^n)-\sin(u^{n-1}), u^{n+1}-u^n\right\rangle_\Omega.
\end{equation}
Denote $\delta u^n = u^n-u^{n-1}$. By using \eqref{4.12t0}, we can rewrite
the left-hand side of \eqref{eq:bdf2int} as
\begin{equation}
{\mathrm {LHS}} = \frac 1 \tau \|\delta u^{n+1}\|^2+ \frac 1{4\tau}\left( \|\delta u^{n+1}\|^2 -  \|\delta u^{n}\|^2+  \|\delta u^{n+1}-\delta u^n\|^2\right).
\end{equation}
Observe that
\begin{align} \label{4.15t0}
\cos u^{n+1} = \cos u^n - \sin u^n \delta u^{n+1} - \frac 12 \cos \xi^n (\delta u^{n+1})^2,
\end{align}
where $\xi^n$ is a function between $u^n$ and $u^{n+1}$. 
By using \eqref{4.15t0}, we rewrite the right-hand side of \eqref{eq:bdf2int} as
\begin{equation}
\begin{aligned}
{\mathrm {RHS}}
&\le  \frac{\ka^2}{2}\|\nabla u^{n}\|^2 +\int_{\Omega} \cos(u^{n})dx - \frac{\ka^2}{2}\|\nabla u^{n+1}\|^2 -
\int_{\Omega}\cos(u^{n+1})  dx\\
&\qquad- \frac 1 2 \left\langle \cos(\xi^n), (\delta u^{n+1})^2\right\rangle_\Omega 
+ \left\langle \sin(u^n)-\sin(u^{n-1}), \delta u^{n+1}\right\rangle_\Omega.
\end{aligned}
\end{equation}
Note that
\begin{equation}\label{eq:3.9}
\begin{aligned}
\left\langle \sin(u^n)-\sin(u^{n-1}), \delta u^{n+1}\right\rangle_\Omega
& \le \|\delta u^n\| \|\delta u^{n+1}\| \le  \frac 1 2 \|\delta u^{n+1} -\delta u^n\|^2
+\frac 3 2 \|\delta u^{n+1}\|^2.
\end{aligned}
\end{equation}
Collecting the estimates, we have
\begin{align}
& \frac 1 \tau \|\delta u^{n+1}\|^2+ \frac 1{4\tau}\left( \|\delta u^{n+1}\|^2 -  \|\delta u^{n}\|^2+  \|\delta u^{n+1}-\delta u^n\|^2\right) \notag \\
\le &\frac{\ka^2}{2}\|\nabla u^{n}\|^2 +\int_{\Omega} \cos(u^{n})dx - \frac{\ka^2}{2}\|\nabla u^{n+1}\|^2 -
\int_{\Omega}\cos(u^{n+1})  dx \notag \\
& \quad -\langle \frac 12 \cos \xi^n - \frac 3 2, (\delta u^{n+1} )^2 \rangle + \frac 12 \| \delta u^{n+1} -\delta u^n \|^2.
\end{align}
Thus we obtain
\begin{equation}\label{ineq:3.13}
\begin{aligned}
& E(u^{n+1}) -E(u^n) + \frac 1{4\tau} \|\delta u^{n+1}\|^2 - \frac 1{4\tau} \|\delta u^{n}\|^2\\
& \leq   -\left\langle \frac1\tau +\frac 1 2\cos(\xi^n) - \frac 3 2, (\delta u^{n+1})^2\right\rangle_\Omega -\left(\frac 1{4\tau} -\frac 1 2\right) \|\delta u^{n+1} -\delta u^n\|^2.
\end{aligned}
\end{equation}
When $0<\tau\leq \frac 1 2$, it is obvious that the right-hand side of \eqref{ineq:3.13} is non-positive so that $\widetilde E(u^{n+1}) \leq \widetilde E(u^{n}) $ holds.
\end{proof}

\section{Numerical experiments}

\begin{example} \label{exam1DACsin}
{\em Consider the 1D sine-Gordon equation
\begin{align}\label{eq:SG}
& \partial_t u = \ka^2\partial_{xx} u + \sin(u), \quad\mbox{on }\mathbb T=[-\pi,\pi],
\end{align}
with $\ka = 0.1$ and $u_0(x)= \pi \sin(x)$.
}
\end{example}

We adopt the first order IMEX scheme \eqref{eq:sch1} to solve this 1D sine-Gordon equation.
For the spatial discretization, we use the pseudo-spectral method with the number of Fourier modes $N = 256$.
On the left-hand side of Figure \ref{fig:1DSG},  we plot the numerical solutions  at $T= 42$ which are computed
using time steps $\tau = 0.1,~2,~2.1$ respectively.
The corresponding energy evolutions are depicted in on the right-hand side of Figure \ref{fig:1DSG}.
It can be observed that when $\tau = 0.1$ and $2$, the energy decays monotonically in time.
However, when $\tau = 2.1$, the energy does not always decay.
This indicates that the time step restriction in Theorem \ref{thm3.1} is optimal.

\begin{figure}[!h]
\centering
\includegraphics[width=0.98\textwidth,clip,trim={0 0.15in 0 0.3in}]{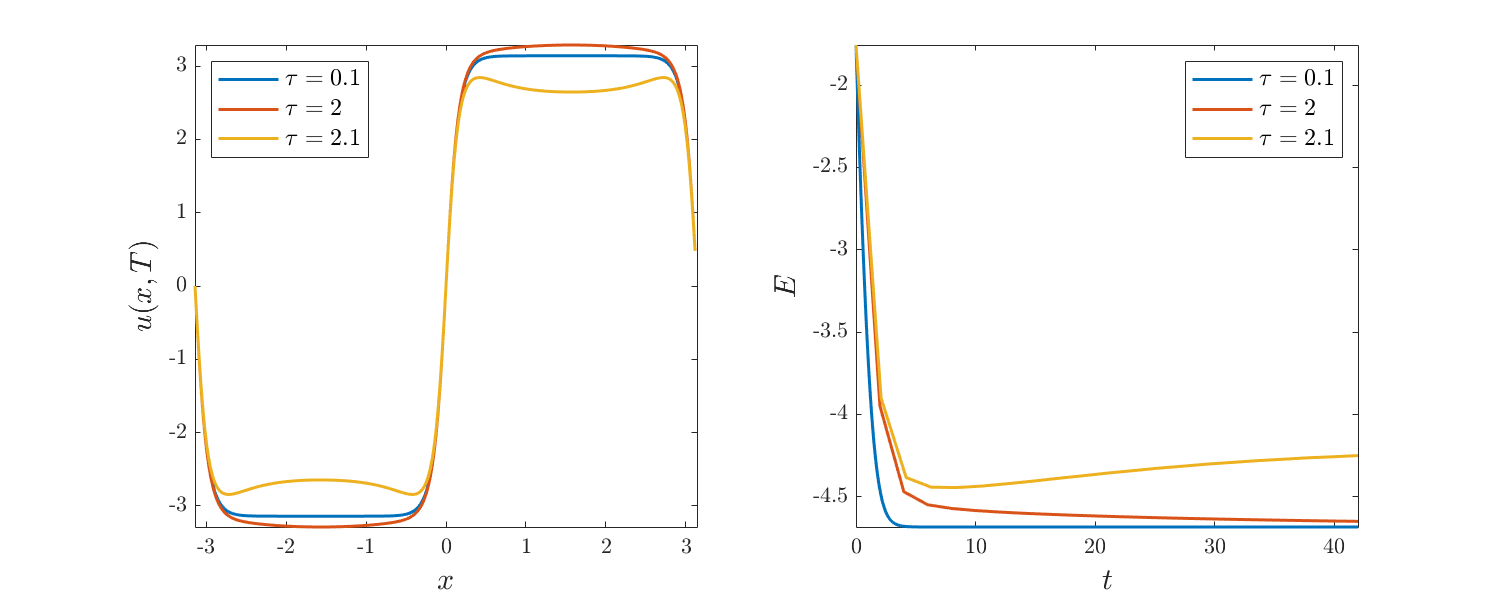}
\caption{\small Numerical solutions $u(x,T)$ and energy evolutions of the 1D sine-Gordon equation with different time steps $\tau = 0.1,~2,~2.1$, where $T = 42$. }\label{fig:1DSG}
\end{figure}

\begin{example} \label{exam2DACsin}
{\em Consider the 2D sine-Gordon equation
\begin{align}\label{eq:SG2d}
& \partial_t u = \ka^2\Delta u + \sin(u), \quad\mbox{on }\mathbb T^2=[-\pi,\pi]^2,
\end{align}
with $\ka = 0.2$ and $u_0(x,y)= \pi\sin(x)\sin(y)$.
}
\end{example}

We compare the numerical solution of \eqref{eq:SG2d} with the standard Allen-Cahn equation with polynomial potential
\begin{align}\label{eq:ac2d}
 \partial_t u = \ka^2\Delta u + u-u^3, \quad\mbox{on }\mathbb T^2=[-\pi,\pi]^2.
\end{align}
This Allen-Cahn equation is solved using the following BDF2 scheme:
\begin{equation}\label{eq:sch3}
\frac{3u^{n+1}-4u^n+u^{n-1}}{2\tau}=\ka^2\Delta u^{n+1}+2 (u^n-(u^n)^3)-(u^{n-1}-(u^{n-1})^3).
\end{equation}
For the spatial discretization, we use the pseudo-spectral method with the number of Fourier modes $N_x\times N_y = 256\times 256$.

\begin{figure}[H]
\centering
\includegraphics[width=0.37\textwidth]{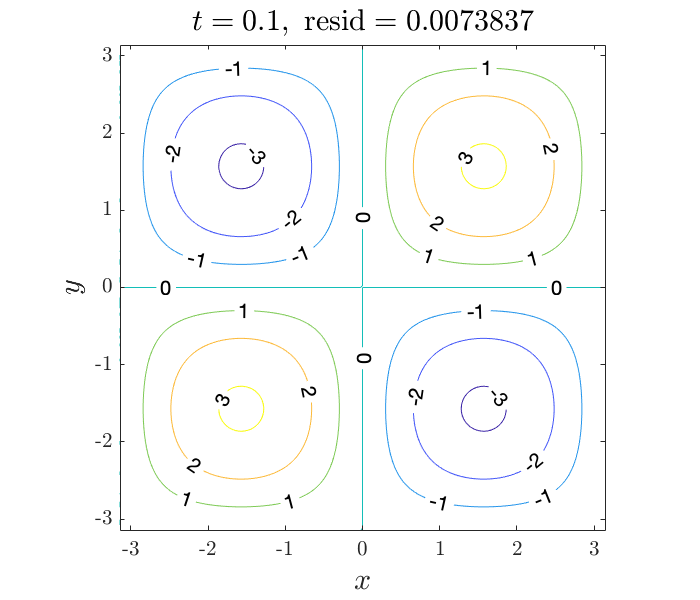}
\includegraphics[width=0.37\textwidth]{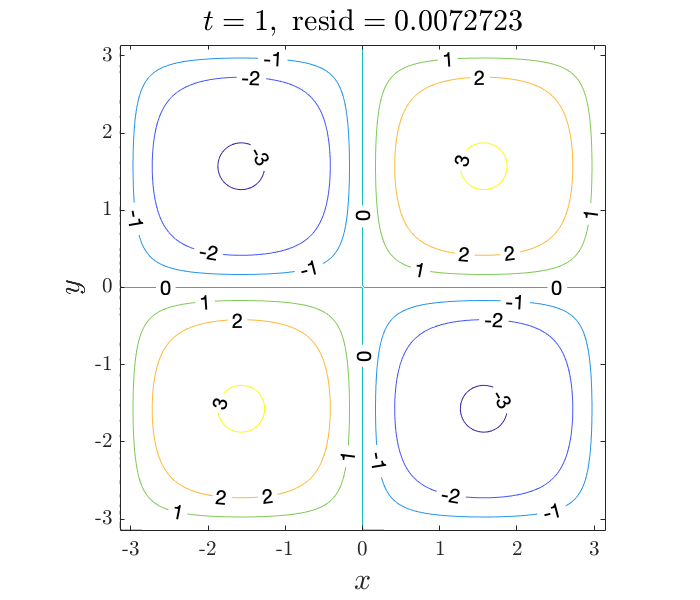}\\
\includegraphics[width=0.37\textwidth]{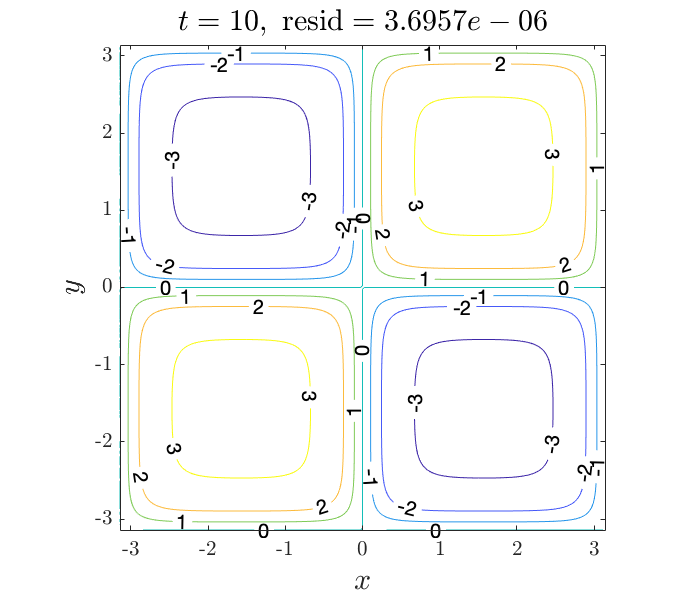}
\includegraphics[width=0.37\textwidth]{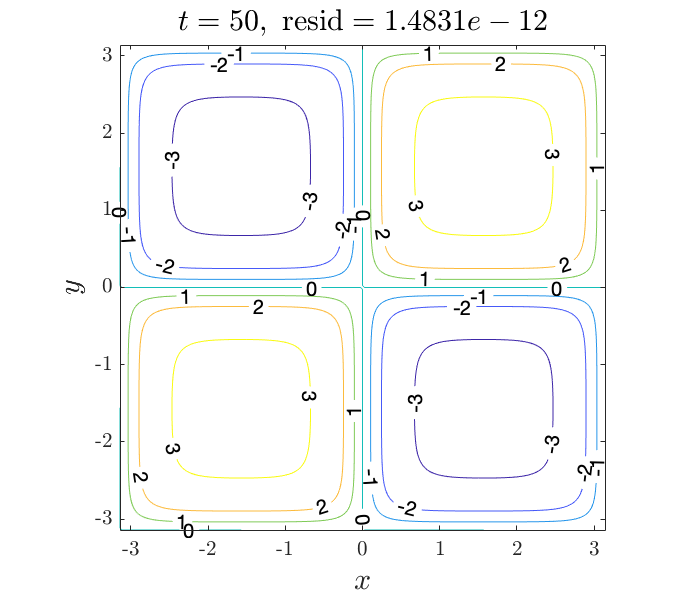}
\caption{\small Example \ref{exam2DACsin}: Dynamics of 2D sine-Gordon equation \eqref{eq:SG2d} using the second-order BDF2 scheme \eqref{eq:sch2} where $\ka = 0.2$, $u_0 =\pi\sin(x)\sin(y)$, $\tau= 0.01,~N_x=N_y = 256$. }\label{fig:2dsg}
\centering
\includegraphics[width=0.37\textwidth]{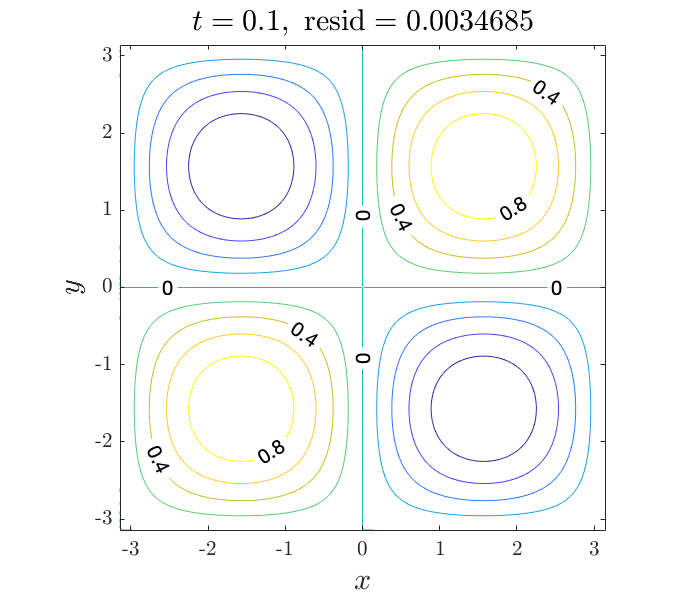}
\includegraphics[width=0.37\textwidth]{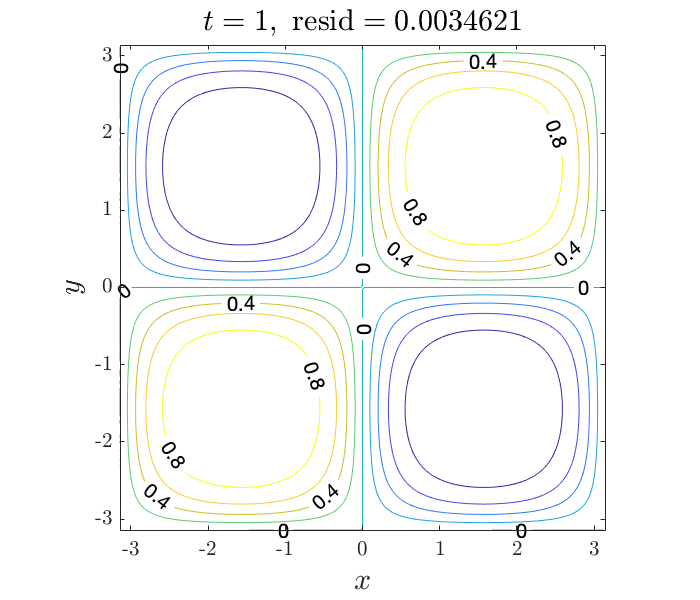}\\
\includegraphics[width=0.37\textwidth]{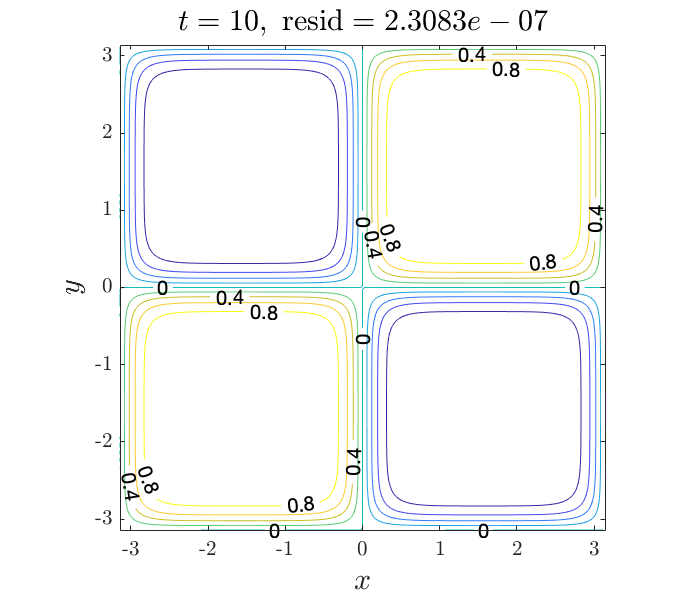}
\includegraphics[width=0.37\textwidth]{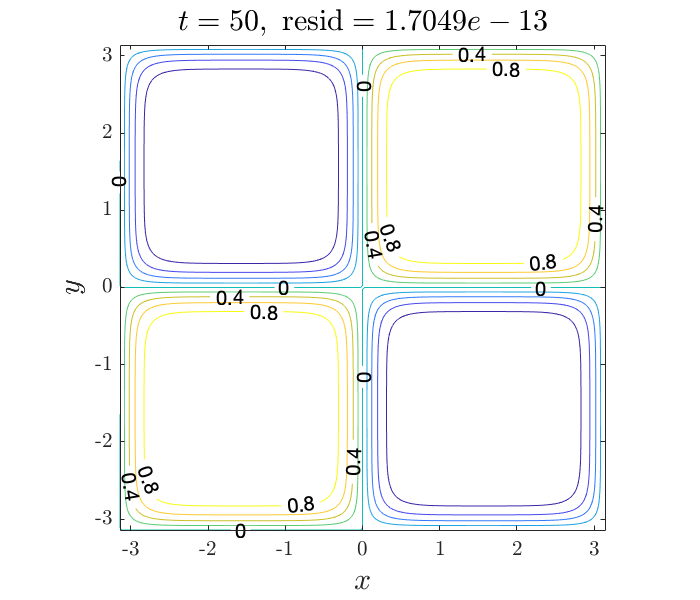}
\caption{\small Example \ref{exam2DACsin}: Dynamics of 2D Allen-Cahn equation \eqref{eq:ac2d} using the second-order BDF2 scheme \eqref{eq:sch3} where $\ka = 0.2$, $u_0 =\sin(x)\sin(y)$, $\tau= 0.01,~N_x=N_y = 256$. }\label{fig:2dac}
\end{figure}

The computed solutions are illustrated  in Figure \ref{fig:2dsg} and \ref{fig:2dac}.
It can be observed that both models exhibit strikingly similar patterns.
The corresponding energy evolutions are presented in Figure \ref{fig:2denergy}.

\begin{figure}[!h]
\centering
\includegraphics[width=0.43\textwidth]{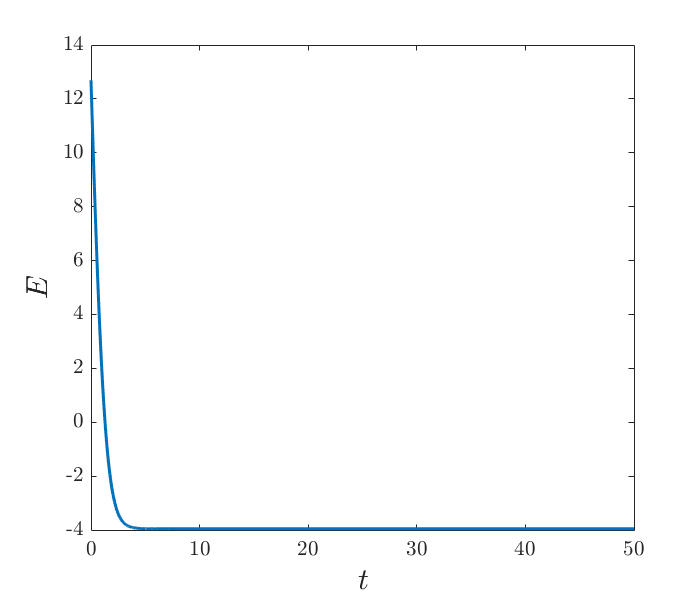}
\includegraphics[width=0.43\textwidth]{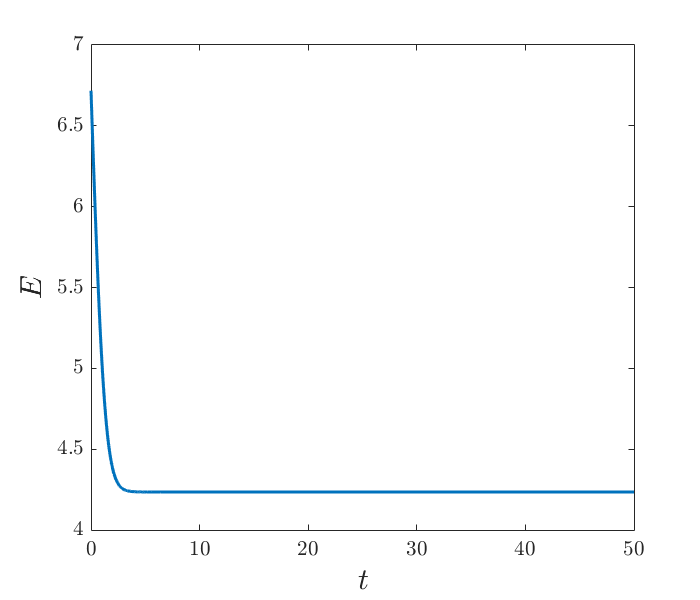}
\caption{\small Energy evolutions for the 2D sine-Gordon simulation in Figure \ref{fig:2dsg} (left) and the 2D Allen-Cahn simulation in Figure \ref{fig:2dac} (right), where $\ka = 0.2$, $\tau= 0.01,~N_x=N_y = 256$. 
Here $u_0 =\pi\sin(x)\sin(y)$ for sine-Gordon and $u_0 =\sin(x)\sin(y)$ for Allen-Cahn. 
The BDF2 pseudo-spectral schemes \eqref{eq:sch2} and \eqref{eq:sch3} are used respectively. }\label{fig:2denergy}
\end{figure}

\section{Concluding remarks}
In this work we introduced a parabolic sine-Gordon (PSG) model which is a special phase field
model with cosine-type potential.  We proved a fundamental maximum principle for the parabolic
sine-Gordon model with periodic boundary conditions in all dimensions.
In the one-dimensional
case we classified all bounded steady states and exhibit some explicit solutions. 
We considered two types of numerical discretization for PSG: one
is first order IMEX, and the other is  BDF2 IMEX. For both schemes we do not use any additional stabilization term. Without appealing to the maximum principle, we rigorously prove the energy stability of the numerical schemes under nearly sharp and quite mild time step constraints. 
By several numerical examples  we demonstrated the striking
similarity of the PSG model with the standard Allen-Cahn equations with double
well potentials.  Due to its inherent benign nonlinear structure, it appears that the PSG model is particularly amenable to $L^{\infty}$-analysis. In prospect we hope the PSG model
will have a ubiquitous presence in phase field simulations.

\frenchspacing
\bibliographystyle{plain}


\end{document}